\documentclass[a4paper,oneside]{amsart}
\usepackage{geometry}
\usepackage{amsmath}

\usepackage{flexisym}
\usepackage{breqn}
\usepackage{mathrsfs} 
\usepackage{eufrak} 
\usepackage{amsfonts} 

\usepackage{comment}
\usepackage{amssymb}
\usepackage{amsmath}
\usepackage{graphicx} 
\usepackage{hyperref}
\usepackage{todonotes}
\hypersetup{colorlinks,linkcolor={blue!50!black},citecolor={blue!50!black},urlcolor={blue!50!black}}

\let\oldlb\(
\let\oldrb\)
\renewcommand{\(}{\begin{dmath*}}
\renewcommand{\)}{\end{dmath*}}
\newtheorem{theorem}{Theorem}[section]
\newtheorem{lemma}[theorem]{Lemma}

\newtheorem{proposition}[theorem]{Proposition}

\newtheorem{definition}[theorem]{Definition}

\newtheorem{question}[theorem]{Question}

\newtheorem{remark}[theorem]{Remark}

\theoremstyle{definition}
\newcommand{\N}{\mathbb{N}}

\newcommand{\Z}{\mathbb{Z}}

\newcommand{\Le}{\mathcal{L}}
\newcommand{\Lep}{\mathscr{L}}
\newcommand{\A}{\mathcal{A}}
\newcommand{\B}{\mathcal{B}}

\title{The strong topological Rokhlin property and Medvedev degrees of SFTs}
\author{Nicanor Carrasco-Vargas}
\address{Faculty of Mathematics and Computer Science, Jagiellonian University, Kraków, Poland}
\email{nicanor.vargas@uj.edu.pl}

\begin{document}
\maketitle
\begin{abstract}
We prove that if a recursively presented group  admits a (nonempty) subshift of finite type with nonzero Medvedev degree then it fails to have the strong topological Rokhlin property. This result simplifies a known criterion and provides new examples of recursively presented groups without this property. 
\end{abstract}
\section{Introduction}
A countable group is said to have the strong topological Rokhlin property (STRP) when it admits a generic continuous action on the Cantor space. This means that  the space of its continuous actions on the Cantor space endowed with the uniform topology has an element  whose topological conjugacy class contains a dense $G_\delta$ set. 

This work concerns the open problem of characterizing countable groups having the STRP, and a sufficient condition for the failure of the STRP. Let us now mention the key works in this context. In the seminal work    \cite{kechris_turbulence_2007} Kechris and Rosendal established the STRP for $\Z$. Hochman \cite{hochman_rohlin_2012} proved that $\Z^d$ fails to have the STRP for $d\geq 2$. Doucha \cite{doucha_strong_2024} generalized these results by showing that free products of finite/cyclic groups have the STRP, and that virtually nilpotent groups that are not virtually $\Z$ fail to have the STRP (also other classes of groups). 

A key observation in  
Hochman's work is that in order to establish the failure of the STRP for $\mathbb{Z}^d$ ($d\geq 2$), it is sufficient to construct a subshift of finite type (SFT) having a sofic factor with the following properties: it has nonzero Medvedev  degree, it equals the union of its minimal subsystems, and these minimal subsystems are not isolated, see \cite[\S 5.1]{hochman_rohlin_2012}.  A related criterion can be found in Doucha's paper  \cite[\S 5.1]{doucha_strong_2024}.

The Medvedev degree of a set quantifies how hard it is to compute at least one of its elements in a recursion-theoretic sense. The underlying philosophy is that one identifies a mathematical problem with the set of its solutions, so the difficulty of the problem corresponds to how hard it is computing at least one solution. In the context of symbolic dynamics and subshifts over finitely generated groups,  Medvedev degrees are a particularly robust notion. Morphisms of subshifts are automatically computable by the Curtis-Lyndon-Hedlund theorem, and from this it follows that Medvedev degrees are invariant for topological conjugacy and non-increasing by factor maps, without requiring the group to be amenable, see \cite{simpson_medvedev_2014,barbieri_medvedev_2024}.

Our main result is the following sufficient condition for the failure of the STRP on a recursively presented group. 
\begin{theorem}\label{main}
    A recursively presented group which admits a (nonempty) SFT with nonzero Medvedev degree  fails to have the STRP.
\end{theorem}
This condition is much easier to verify than the one established in Hochman's work \cite{hochman_rohlin_2012}. Indeed, in the case of $\mathbb{Z}^d$ ($d\geq 2$) the existence of SFTs with nonzero Medvedev degree is known since the 1970s \cite{hanf_nonrecursive_1974,myers_nonrecursive_1974}.

As a consequence of Theorem \ref{main} and existing results about Medvedev degrees of SFTs, we obtain several new examples of recursively presented groups not having the STRP. This includes Baumslag-Solitar groups \cite{aubrun_tiling_2013}, Lamplighter groups \cite{bartholdi_shifts_2024}, groups quasi-isometric to the hyperbolic plane \cite{barbieri_medvedev_2024}, branch groups \cite{barbieri_medvedev_2024}, direct products of two infinite groups \cite[Theorem 5.8]{barbieri_medvedev_2024} (this was previously known for direct products of three infinite groups \cite[Theorem C]{doucha_strong_2024}), and groups admitting a simulation theorem  \cite[\S 5.5]{barbieri_medvedev_2024}.

The proof of Theorem \ref{main} strongly relies on Doucha's result   \cite[Theorem A]{doucha_strong_2024}, which shows that a  finitely generated group has the STRP if and only if projectively isolated subshifts are dense in the space of subshifts on alphabet $\A$ for each finite alphabet $\A$. Projectively isolated subshifts are a class of sofic subshifts introduced in \cite{doucha_strong_2024}, see Definition \ref{def:projectively-isolated}. We prove that if an SFT has nonzero Medvedev degree, then it has a neighborhood in this topological space without projectively isolated subshifts. The technical implementation of our argument  relies on some results from \cite{barbieri_effective_2025}. 
\section{Preliminaries}
We write $A\Subset B$ when $A$ is finite and $A\subset B$, and we denote by $A^B$ the set of functions from $B$ to $A$. 
\subsection{Finitely generated groups}
Let $G$ be a finitely generated infinite group with identity element $1_G$, and let $S$ be a finite generating set closed under inverses. The set of words on $S$ is denoted by $S^\ast$. Given $w\in S^\ast$ we denote by $\underline{w}$ the corresponding group element. Let $F(S)$ be the free group generated by $S$, and let $\pi\colon F(S)\to G$ be the group homomorphism which sends a reduced word $w$ to $\underline w$.

Let $W_n\Subset S^\ast$ be the set of words of length at most $n$, and let $B_n\Subset G$ be the set $\{\underline{w} : w\in W_n\}$, $n\geq 1$. We write $gF=\{gf : f\in F\}$ for $g\in G$ and $F\Subset G$.

In most of our statements we will assume that $G$ is recursively presented, meaning that the word problem $\{w\in S^\ast : \underline{w}=1_G\}$ is recursively enumerable. We also recall that $G$ has decidable word problem when $\{w\in S^\ast : \underline{w}=1_G\}$ is decidable.  
\subsection{Subshifts}
We follow \cite{ceccherini-silberstein_cellular_2018}. Let $\A$ be a finite set called alphabet, and endow $\A^G$ with the metric
\[d(x,y)=\inf\{2^{-n} : x|_{B_n}=y|_{B_n}\}, \ \ \ x,y\in \A^G.\]
Elements in $\A^G$ are called configurations. We consider the shift action $G\curvearrowright \A^G$ given by $gx(h)=x(g^{-1}h)$. A subshift is a nonempty closed and shift-invariant subset of $\A^G$. 

Given two subshifts $X\subset \A^G$ and $Y\subset \B^G$, a function $\phi\colon Y\to X$ is called a morphism if it is continuous and $g\phi(x)=\phi(gx)$ for all $g\in G$ and $x\in Y$. The Curtis-Lyndon-Hedlund theorem states that this is equivalent to the existence of a finite set $F\Subset G$ and a local function $\ell \colon \B^F\to \A$ such that $\phi(x)(g)=\ell((g^{-1}x)|_F)$ for every $x\in X$ and $g\in G$. If $\phi$ is surjective then we say that it is a factor map and that $X$ is a factor of $Y$.

A pattern with support $F\Subset G$ is a function $p\colon F\to \A$. We write $p\sqsubset x$ for $x\in \A^G$ if $x|_F=p$.  For a subshift $X$ we write $\Le_F(X)=\{p\in \A^F : p\sqsubset x \text{ for some }x\in X\}$. The language of $X$ is $\Le(X)=\bigcup_{F\Subset G}\Le_F(X)$. We also write $\Le^c(X)=\Le(\A^G)\smallsetminus \Le(X)$.

We say that a pattern $p$ appears in a configuration $x\in\A^G$ when $p\sqsubset gx$ for some $g\in G$. A subshift can be defined by forbidding a set of patterns  $\mathcal F\subset \Le(\A^G)$ as  
 \[X_{\mathcal F}=\{x\in \A^G : \text{ no pattern from $\mathcal F$ appears in $x$}\}\]

A subshift of finite type (SFT) is one which can be defined by a finite set of forbidden patterns, and a subshift is sofic when it  is the factor of an SFT.

\subsection{The metric space of subshifts}
We consider the space of subshifts
\[S(\A^G)=\{X\subset \A^G : \text{$X$ subshift }\}\] endowed with the metric 
\[D(X,Y)=\inf\{2^{-n} : \Le_{B_n}(X)=\Le_{B_n}(Y)\}, \ \  \ X,Y\in S(\A^G).\]
One easily verifies that $D$ is the Hausdorff distance associated to $d$.

\begin{remark}
SFTs are dense in $(S(\A^G),D)$.
\end{remark}
\begin{definition}[\cite{doucha_strong_2024}]\label{def:projectively-isolated}
A subshift $X\in S(\A^G)$ is projectively isolated if there is an alphabet $\B$, an open set $U$ in $S(\B^G)$, and a morphism of subshifts $\phi\colon \B^G\to \A^G$, so that $\phi(Y)=X$ for every $Y\in U$. 
\end{definition}
\begin{remark}
    A projectively isolated subshift is sofic. Indeed, since SFTs are dense in $S(\B^G)$, we can pick $Y\in U$ in the previous definition to be an SFT, and then the relation $\phi(Y)=X$ shows that $X$ is the factor of an SFT. 
\end{remark}
\subsection{Computability}
A set of integers $P\subset \N$ is recursively enumerable if there is a Turing machine which on input $n\in\N$ halts if and only if $n\in P$. This is equivalent to the existence of a Turing machine which outputs an infinite list whose range is $P$. We say that $P$ is co-recursively enumerable if $\N\smallsetminus P$ is recursively enumerable, and that $P$ is decidable if it satisfies both conditions. These notions extend to sets of words, sets of finite sets, sets of pattern presentations (defined below), and other subsets of countable sets through computable enumerations. 

In order to define computability notions for subshifts over  finitely generated groups without assuming them to have decidable word problem we introduce the following artifact. A pattern presentation with support $F\Subset S^\ast$ is a function $p\colon F\to \A$. We write $p\sqsubset x$ for a configuration  $x\in \A^G$ when $x(\underline{w})=p(w)$ for every $w\in F$, and we say that $p$ appears in $x$ whenever $p\sqsubset gx$ for some $g\in G$. We say that $p$  is consistent when $p(u)=p(v)$ for every $u,v\in F$ with $\underline{u}=\underline{v}$. Provided $p$ is consistent, we can define an actual pattern $q$ with support $\{\underline{w} : w\in F\}\Subset G$ by the relation $q(\underline{w})=p(w)$. In this situation we say that $p$ is a presentation of $q$.   
 
For a subshift $X$ and  $F\Subset S^\ast$ we write $\mathscr{L}_F(X)=\{ p \in \A^F : \  p\sqsubset x \text{ for some } x\in X\}$. We define $\Lep(X)=\bigcup_{F\Subset S^\ast}\Lep_F(X)$, and  $\Lep^c(X)=\Lep(\A^G)\smallsetminus \Lep(X)$.  Therefore $\Lep(X)$ is the set of presentations of patterns in the language $\Le(X)$ of $X$. 

We say that a subshift $X$  is effective when $\Lep^c(X)$ is recursively enumerable, and we say that the language $\Le(X)$ of $X$ is decidable when the corresponding set of presentations $\Lep(X)$ is decidable. Effective subshifts are closed by factor maps. If the group is recursively presented, then both SFTs and sofic subshifts are effective. 

We note that there are several definitions of effective subshifts in the literature, but they all coincide for recursively presented groups. The definition given here is the most restrictive one, and it equivalent to the requirement that the shift action $G\curvearrowright X$  is an effective dynamical system in the sense of being conjugate to an action by computable homeomorphisms on a recursively compact subset of $\{0,1\}^\N$. We refer the reader to  \cite[\S 7]{barbieri_effective_2025} for further details on effective subshifts. 

We also mention that in the case of a group with decidable word problem, we can computably identify $\Le(\A^G)$ with $\N$ and work directly with patterns instead of pattern presentations. In this case a subshift is effective when $\Le^c(X)$ is recursively enumerable, and has decidable language when $\Le(X)$ is decidable.


The zero Medvedev degree is denoted $0_{\mathfrak M}$, and the Medvedev degree of a subshift $X$ is denoted $m(X)$. Medvedev degrees are non-decreasing by inclusions, in the sense that $X\subset Y$ implies $m(X)\geq m(Y)$.  A subshift $X$ satisfies  $m(X)=0_{\mathfrak M}$ if and only if it has a computable configuration. By a computable configuration we mean an element $x\in \A^G$ for which there is a Turing machine which on input $w\in S^\ast$ halts and outputs $x(\underline{w})\in\A$.

Defining Medvedev requires introducing computable functions over $\{0,1\}^\N$,  so we omit the full definition and refer the reader to \cite[\S 3]{barbieri_medvedev_2024} for a more detailed exposition. We will only need the properties mentioned in the previous paragraph. 
\section{Proof of Theorem \ref{main}}
\begin{proposition}\label{prop:projectively-isolated-decidable-language}
A projectively isolated subshift on a recursively presented group has decidable language. 
\end{proposition}
\begin{proof}
    Let $X\subset \A^G$ be a projectively isolated subshift. Our goal is showing that $\Lep(X)$ is decidable. Since a projectively isolated subshift is sofic and sofic subshifts are effective, it follows that $\Lep^c(X)$ is recursively enumerable. We now prove that $\Lep(X)$ is also recursively enumerable. 

    Fix a finite alphabet $\B$, an open set $U\subset S(\B^G)$, and a morphism $\phi\colon \B^G\to \A^G$ as in the definition of projectively isolated subshift associated to $X$. Since $U$ is open and SFTs are dense in $S(\B^G)$, we can pick an SFT $Y\in U$. Since $U$ is open, we can find a radius $\epsilon$ so that the ball in the metric $D$ with radius $\epsilon$ and center $Y$ is contained in $U$. Fix $k\in\N$ so that $2^{-k}<\epsilon$. It follows then from the definition of $D$ that every subshift $Z\in S(\B^G)$ with $\Le_{B_k}(Z)=\Le_{B_k}(Y)$ satisfies $Z\in U$. 

    In order to prove that $\Lep(X)$ is recursively enumerable, we consider a pattern presentation $p\colon F\to \A$ with support $F\Subset S^\ast$, and describe an algorithmic procedure that halts if and only if $p\in\Lep(X)$. For this we define two subshifts $X_p$ and $Y_p$ which depend on $p$.
    \[X_p=\{x\in X : \text{$p$ does not appear in $x$}\},\] 
    \[Y_p=\{y \in Y : \phi(y)\in X_p\}.\]
    
    Observe that $Y_p\subset Y$ and thus $\Lep_{W_k}(Y_p)\subset \Lep_{W_k}(Y)$. The number $k$ depends on $X$ and not on $p$, so we can ``hard-code'' it in the next algorithm, together with the finite set $\Lep_{W_k}(Y)$. 
    
    The key observation is that $p$ belongs to $\Lep(X)$ if and only if $\Lep_{W_k}(Y_p)$ is properly contained in $\Lep_{W_k}(Y)$. This is equivalent to the assertion ($p\not\in\Lep(X)$) $\iff$ $(\Lep_{W_k}(Y_p)=\Lep_{W_k}(Y))$. Let us prove this equivalence. 
    \begin{itemize}
        \item $(\Rightarrow)$. Suppose that $p\not\in\Lep(X)$. Then it follows from the definition of $X_p$ that $X_p=X$, which implies $Y_p=Y$, and therefore $\Lep_{W_k}(Y_p)=\Lep_{W_k}(Y)$. 
        \item $(\Leftarrow)$. Suppose that $\Lep_{W_k}(Y_p)=\Lep_{W_k}(Y)$. This assumption implies that we also have $\Le_{B_k}(Y_p)=\Le_{B_k}(Y)$.

        We first prove that $\phi(Y_p)=X$. Here we use that $X$ is projectively isolated, and our choice of $k$. The equality $\Le_{B_k}(Y_p)=\Le_{B_k}(Y)$ implies that $Y_p$ belongs to $U$. Then $\phi(Y_p)=X$ by definition of $U$. 
        
        Next we prove that $\phi(Y_p)=X_p$. Let $\psi\colon Y\to X$ be the restriction of $\phi$. Thus $\psi$ is a surjective function. It is clear that $\psi(Y_p)=\phi(Y_p)$, so it suffices to prove that $\psi(Y_p)=X_p$. We also have $\psi(\psi^{-1}(X_p))=X_p$ by an elementary property of surjective functions. But since $\psi^{-1}(X_p)=Y_p$, it follows that $\psi(Y_p)=X_p$ as claimed.   
        
        We proved that $\phi(Y_p)=X$ and also that $\phi(Y_p)=X_p$, so we can conclude that $X=X_p$. The last equality implies $p\not\in\Lep(X)$.
    \end{itemize} 
    We have proved our claim that $p\in\Lep(X)$ if and only if $\Lep_{W_k}(Y_p)\subsetneq\Lep_{W_k}(Y)$.
    
    In order to finish the proof, it suffices to show that the relation $\Lep_{W_k}(Y_p)\subsetneq\Lep_{W_k}(Y)$ is recursively enumerable. For this we will use the fact that $Y_p$ is effective uniformly on $p$, meaning that there is an algorithm which given $p\colon F\to \A$ ($F\Subset S^\ast$) and $q\colon E\to \B$ ($E\Subset S^\ast$) halts if and only if $q\in \Lep^c(Y_p)$. This will be proved separately in Lemma \ref{lemma} below. We will also use the fact that $\Lep_{W_k}(Y)$ is finite and does not depend on $p$. Let $\{q_1,\dots,q_m\}=\Lep_{W_k}(Y)$. In order to determine whether $\Lep_{W_k}(Y_p)\subsetneq\Lep_{W_k}(Y)$, it suffices to simultaneously run $m$ instances of the algorithm in the previous paragraph, one for each $q_i$, $i=1,\dots,m$. If some of these $m$ processes halts to conclude that $q_i\in\Lep^c_{W_k}(Y_p)$, then we stop the whole process to conclude $\Lep_{W_k}(Y_p)\subsetneq\Lep_{W_k}(Y)$ and consequently $p\in\Lep(X)$. \end{proof}
\begin{lemma}\label{lemma}
    Let $G$ be a finitely generated group and let $\phi\colon \B^G\to \A^G$ be a morphism. Let $X\subset\A^G$ and  $Y\subset \B^G$ be effective subshifts satisfying $X=\phi(Y)$. Let $p\colon F\to \A$ be a pattern presentation with support $F\Subset S^\ast$, and define
    \[X_p=\{x\in X : \text{$p$ does not appear in $x$}\},\] 
    \[Y_p=\{y \in Y : \phi(y)\in X_p\}.\]
    Then $\Lep^c(Y_p)$ is recursively enumerable uniformly on $p$. That is, $\{(q,p) : q\in \Lep^c(Y_p)\}$ is recursively enumerable.  
\end{lemma}
\begin{proof}

Recall that $\pi\colon F(S)\to G$ is the group homomorphism which sends a reduced word $w$ to $\underline{w}$. Given $z\in \A^G$ we define $\hat z\in\A^{F(S)}$ by $\hat z(g)=z(\pi(g))$, $g\in F(S)$. Given a subshift $Z\subset \A^G$ we define $\hat{Z}=\{\hat z : z\in Z\}$. Observe that we have the set-theoretical equality $\Lep(Z)=\Lep(\hat Z)$. Furthermore, the association $Z\to\hat Z$ preserves\footnote{We remark that this is not true for arbitrary finitely generated groups if we take a less restrictive definition of effective subshift \cite[Proposition 7.1]{barbieri_effective_2025}, but with the definition used here, this is true in general.} the property of being an effective subshift  \cite[\S 3.3 and \S 7] {barbieri_effective_2025}. It follows from these facts that it suffices to prove the statement for a finitely generated free group. 

For the next argument we will only need to assume that our group $G$ has decidable word problem (free groups have decidable word problem). Thus we can work directly with patterns instead of pattern presentations.  We prove that there exists an algorithm which given $p\colon F\to \A$ ($F\Subset S^\ast$) and $q\colon E\to\B$ ($E\Subset G$), halts if and only if $q\in \Le^c(Y_p)$. 

We need to define translations of patterns. Given $g\in G$ the translation $gq$ of the pattern $q$ is defined as the pattern with support $gE$ and satisfying $gq(h)=q(g^{-1}h)$, $h\in gE$. Next we  define the natural extension of $\phi$ to a function over patterns, which we denote by $\Phi$. By the Curtis-Lyndon-Hedlund Theorem we can find a finite set $T\Subset G$ and a local function $\ell\colon \B^T\to \A$ such that $\phi(x)(g)=\ell((g^{-1}x)|_T)$ for $g\in G$ and $x\in \B^G$. We define $\Phi(q)$ as the pattern with support $\{g\in G : T\subset g^{-1}E\}$ and satisfying $\Phi(q)(g)= \ell((g^{-1}q)|_T)$. It is clear that $\Phi\colon \Le(\B^G)\to\Le(\A^G)$ is a computable function.   

A standard argument using the compactness of $\B^G$ shows that the following two conditions are equivalent. 
\begin{enumerate}
    \item $q\in \Le^c(Y_p)$
    \item There exists  $R\Subset G$ containing $E$ such that ($\forall r\in \{s\in \B^R : s|_E=q\}$)($r\in \Le^c(Y)$  or $\Phi(r)\in \Le^c(X_p)$). 

\end{enumerate}
To finish the argument it suffices to note that (2) is recursively enumerable as a predicate of $p$ and $q$.  This follows from the fact that $\Le^c(Y)$ is recursively enumerable, that $\Le^c(X_p)$ is recursively enumerable uniformly on $p$ (see the proof of Proposition 2.1 in  \cite{aubrun_notion_2017}), and that $\Phi$ is computable. Thus we have proved the statement in the special case of a group with decidable word problem.


Let us provide an alternative and more abstract proof of the statement for a group with decidable word problem, using the language of computable Polish metric spaces. We first argue that $\phi^{-1}(X_p)$ is an effective subshift. We use the following basic facts, which can be found in \cite{barbieri_effective_2025} (see also \cite{barbieri_medvedev_2024,carrasco-vargas_subshifts_2024}).
\begin{itemize}
    \item A fullshift $\mathcal C^G$ can be naturally endowed with a computable metric space structure.
    
    \item A function between two computable metric spaces is computable exactly when the pre-image of every effectively closed set is effectively closed (effectively closed sets are the effective counterpart of closed sets, they do not need to be shift-invariant).
    \item $\phi\colon\B^G\to\A^G$ is a computable function thanks to the Curtis-Lyndon-Hedlund Theorem.
    \item A subshift $Z\subset \mathcal C^G$ is effective if and only if it is an effectively closed subset of $\mathcal C^G$.  
\end{itemize}
Thus we can reason as follows: since $X_p$ is an effective subshift, it is an effectively closed set, hence its computable pre-image $\phi^{-1}(X_p)$ is effectively closed, and therefore it is an effective subshift. But clearly $Y_p=Y\cap \phi^{-1}(X_p)$ and the intersection of two effective subshifts is still effective, so we conclude that $Y_p$ is an effective subshift. Finally, all the intermediate steps depend on $p$ in a computable manner, so we obtain an algorithm which given $q$ and $p$ halts if and only if $q\in \Lep^c(Y_p)$. 
\end{proof}
It is well-known that having decidable language is a stronger property than having zero Medvedev degree. This is a classic result in the context of $\Pi_1^0$ classes. We have not found this statement in the literature in a form that can be directly applied to subshifts over recursively presented groups, so we include a general argument for completeness. 
\begin{proposition}\label{decidable-language-implies-zero-medvedev-degree}
A subshift with decidable language has zero Medvedev degree (on every finitely generated group). 
\end{proposition}
\begin{proof}
Let $X\subset\A^G$ be a subshift and let $\hat X\subset \A^{F(S)}$ the subshift defined in the first paragraph of the proof of Lemma \ref{lemma}. Then $m(X)=m(\hat X)$ by definition \cite[\S 3]{barbieri_medvedev_2024}, so it suffices to show that $m(\hat X)=0_{\mathfrak M}$.

Given a pattern $p\colon F\to \A$ with support $F\Subset F(S)$ let us denote by $[p]$ the cylinder set $\{x\in\A^{F(S)} : p\sqsubset x\}$. The hypothesis that $\Lep(X)$ is decidable is equivalent to the statement that given a pattern $p\in\Le(\A^{F(S)})$, the relation $[p]\cap \hat X\ne\emptyset$ is decidable. Then Proposition 2.3.2 in \cite{hoyrup_genericity_2017} proves that $\hat X$ has a dense subset of computable configurations (to apply this result one first endows $\A^{F(S)}$ with the computable metric space structure defined in \cite[\S 3]{barbieri_effective_2025}). It follows that $\hat X$ has at least one computable configuration, and therefore $m(\hat X)=0_{\mathfrak M}$. 
\end{proof}

\begin{proposition}\label{SFT-with-nonzero-degree-has-neighborhood-without-projectively-isolated-subshifts}
    Let $G$ be a recursively presented group and let $\A$ be a finite alphabet. Every SFT in $S(\A^G)$ with nonzero Medvedev degree has a neighborhood without projectively isolated subshifts. 
\end{proposition}
\begin{proof}
Let $S(X)=\{Y\in S(\A^G) : Y\subset X\}$. The fact that $X$ is SFT implies that $S(X)$ is an open subset of $S(\A^G)$. See Proposition 6.3 in \cite{carrasco-vargas_subshifts_2024} or Lemma 2.3 in \cite{pavlov_structure_2023}. Since Medvedev degrees are non-decreasing by inclusions \cite[Proposition 3.1] {barbieri_medvedev_2024}, every subshift $Y\in S(X)$ satisfies $m(Y)\geq m(X)$, and in particular $m(Y)>0_{\mathfrak M}$. On the other hand, it follows from 
Proposition \ref{prop:projectively-isolated-decidable-language} and Proposition \ref{decidable-language-implies-zero-medvedev-degree} that a projectively isolated subshift has zero Medvedev degree. Thus $S(X)$ contains no projectively isolated subshift. 
\end{proof}
\begin{proof}[Proof of Theorem \ref{main}] It follows from 
\cite[Theorem A]{doucha_strong_2024} that $G$ has the STRP if and only if projectively isolated subshifts are dense on $S(\A^G)$ for every finite alphabet $\A$. If $G$ admits an SFT $X$ with $m(X)>0_{\mathfrak M}$ on alphabet $\A$, then it follows from Proposition \ref{SFT-with-nonzero-degree-has-neighborhood-without-projectively-isolated-subshifts} that $S(\A^G)$ has an open set without projectively isolated subshifts. 
\end{proof}
\section{Further remarks and questions}
The connection between Medvedev degrees and the STRP seems intriguing at first sight. However, there is a rather simple explanation to this fact related to abstract properties of Medvedev degrees: 
\begin{enumerate}
    \item Medvedev degrees are non-increasing by inclusions.
    \item A projectively isolated subshift has minimal Medvedev degree (assuming that the group is recursively presented). 
\end{enumerate} Any function over $S(\A^G)$ with these properties would suffice to prove  Theorem \ref{main}. Indeed, let us state a more precise result.
\begin{proposition}Let $G$ be a finitely generated group and let $\A$ with $|\A|\geq 2$. Let $(\mathcal P,\leq)$ be a partially ordered set with a minimal element $0_{\mathcal P}$, and let $\mathcal I\colon S(\A^G)\to \mathcal P$ be a function with the following two properties. 
\begin{enumerate}
    \item $\mathcal I$ is non-decreasing by inclusions. That is, $X\subset Y\Rightarrow \mathcal I(X)\geq \mathcal I(Y)$ for $X,Y\in S(\A^G)$.
    \item $\mathcal I(X)=0_{\mathcal P}$ for every projectively isolated  $X\in S(\A^G)$.
\end{enumerate} If there exists an SFT $X\in S(\A^G)$ such that $\mathcal I(X)\ne 0_{\mathcal P}$, then $S(X)=\{Y\in S(\A^G) : Y\subset X\}$ is a neighborhood of $X$ without projectively isolated subshifts, and $G$ does not have the STRP.
\end{proposition}
\begin{proof}
    This follows from \cite[Theorem A]{doucha_strong_2024} and the same reasoning used in Proposition  \ref{SFT-with-nonzero-degree-has-neighborhood-without-projectively-isolated-subshifts}.
\end{proof}
This observation explains why Medvedev degrees have proven to be an useful tool in relation to the STRP. At least for amenable groups, topological entropy and other related popular invariants have the opposite behavior under inclusions. 

It is natural to ask  whether the techniques used in this work can be applied beyond recursively presented groups. As a test case we propose the following question.
\begin{question}
Is it true that every direct product of two finitely generated infinite groups fails to have the STRP? 
\end{question}A natural direction would be to define a relativized version of Medvedev degrees in which the word problem of the group is used as oracle, and then replicate the proof of \cite[Theorem 5.6]{barbieri_medvedev_2024} to obtain SFTs  whose Medvedev degree relative to the word problem of the group is nontrivial.

\section*{Acknowledgements} I am grateful to Michal Doucha for a fruitful correspondence exchange which derived in a proof of Theorem \ref{main}, and to Sebasti\'an Barbieri for kindly proofreading an initial version of this work. The author was supported by a grant from the Priority Research Area SciMat under the Strategic Programme Excellence Initiative at Jagiellonian University.
\bibliographystyle{abbrv}
\bibliography{references}
\end{document}